\documentclass[11pt]{article}

\usepackage[margin=1.3in]{geometry}
\usepackage{mathtools}
\usepackage[dvips]{graphics}
\usepackage{epsfig}
\usepackage{multirow}
\usepackage{array}
\usepackage{extarrows}
\usepackage{graphicx}
\mathtoolsset{showonlyrefs}

\usepackage[symbol]{footmisc}

\usepackage{comment}
\usepackage[titletoc,title]{appendix}
\numberwithin{equation}{section}
\usepackage{amsthm}
\usepackage{enumerate}
\usepackage{amsmath}
\newtheorem{proposition}{Proposition}

\usepackage{amssymb,amsmath}
\usepackage{latexsym}
\usepackage[usenames]{xcolor}
\usepackage{pgf}
\usepackage[numbers]{natbib}
\usepackage{hyperref}
\hypersetup{
    colorlinks,%
    citecolor=black,%
    filecolor=black,%
    linkcolor=black,%
    urlcolor=black
}
\usepackage{subfig}
\numberwithin{equation}{section}

\newtheorem{theorem}{ \noindent T{\footnotesize HEOREM}}

\newtheorem{lemma}{ \noindent L{\footnotesize EMMA}}[section]

\newtheorem{conjecture}{ \noindent C{\footnotesize ONJECTURE}}

\begin{document}

\title{A note on correlation inequalities for regular increasing families}
\author{Yiming Chen$^\ddagger$, Guozheng Dai$^*$}

\date{}
\maketitle

\begin{abstract}
This paper establishes quantitative correlation inequalities between monotone events and structured threshold objects in both the discrete cube and Gaussian space. We prove that for any increasing balanced family, there exists a linear threshold function yielding a covariance lower bound of $c \frac{\log n}{\sqrt{n}}$, and extend this principle to halfspaces in Gaussian space. These results verify the conjectures of Kalai, Keller, and Mossel regarding optimal correlation bounds for linear threshold functions and their Gaussian analogues.

\end{abstract}

\footnotetext[3]{School of Mathematical Sciences, Peking University, Beijing, China, ymchenmath@math.pku.edu.cn. }
\footnotetext[1]{Department of Mathematics, Hong Kong University of Science and Technology, Clear Water Bay, Kowloon, Hong Kong, guozhengdai@ust.hk.}

\bigskip

\section{Introduction}\label{sec:intro}

Correlation inequalities lie at the heart of combinatorics and probability theory, quantifying how much two families of sets, functions, or random variables tend to align. Beyond bare probability estimates, they control intersection patterns of set families, phase transitions in random graphs, and noise stability of monotone systems, making them indispensable for analysing discrete structures and monotonic random systems.

Let \(\Omega_n\) denote the discrete cube \(\{0,1\}^n\), whose \(2^n\) points are all binary strings of length \(n\).  
We identify each vector \(x = (x_1, \dots, x_n) \in \Omega_n\) with the subset  
\[ A=\{i \in [n] : x_i = 1\} \subseteq [n] = \{1, 2, \dots, n\};\]  
thus, the \(i\)-th coordinate \(x_i\) indicates whether element \(i\) is present in \(A\).
A family \(\mathcal{A} \subseteq \Omega_n\) is \emph{increasing} if for every \(S \in \mathcal{A}\) and every \(T \supseteq S\) (where \(S,T\) are subsets of \([n]\)) one has \(T \in \mathcal{A}\); equivalently, its indicator function \(\mathbf{1}_{\mathcal{A}} \colon \Omega_n \to \{0,1\}\) is non-decreasing with respect to the natural partial order on \(\Omega_n\).

A classical result of Harris \cite{H1960} asserts that any two increasing families $\mathcal{A},\mathcal{B}\subseteq\Omega_n$ are non-negatively correlated:
\[
\mathrm{Cov}_\mu(\mathbf{1}_\mathcal{A},\mathbf{1}_\mathcal{B})=\mu(\mathcal{A}\cap\mathcal{B})-\mu(\mathcal{A})\mu(\mathcal{B})\geq 0,
\]
where $\mu$ always denotes the uniform measure on $\Omega_n$ throughout the paper.

Talagrand \cite{T1996} gave a quantitative lower bound on this covariance in terms of the \emph{influences} of the individual coordinates on $\mathcal A$ and $\mathcal B$.
For a family $\mathcal A\subseteq\Omega_n$, the influence of coordinate $k$ under the uniform measure $\mu$ is defined as
\begin{align}\label{Eq_influence}
I_k(\mathcal A)=2\mu\!\bigl(\{x\in\mathcal A:x\oplus e_k\notin\mathcal A\}\bigr),
\end{align}
where $x\oplus e_k$ denotes the vector obtained from $x$ by flipping the $k$-th coordinate.  
The \emph{total influence} of $\mathcal{A}$ is
\[
I(\mathcal{A})=\sum_{k=1}^n I_k(\mathcal{A}).
\]
We also write
\[
\mathcal{W}_1(\mathcal{A},\mathcal{B})=\sum_{i=1}^n I_i(\mathcal{A})I_i(\mathcal{B}).
\]
Talagrand \cite{T1996} proved that, 
for increasing families $\mathcal{A},\mathcal{B}\subseteq\Omega_n$,  
\begin{align}\label{Eq_Talagrand_result}
\mathrm{Cov}_\mu(\mathbf{1}_\mathcal{A},\mathbf{1}_\mathcal{B})\geq c\,\frac{\displaystyle\mathcal{W}_1(\mathcal{A},\mathcal{B})}{\log\!\Bigl(e\Big/\displaystyle \mathcal{W}_1(\mathcal{A},\mathcal{B})\Bigr)},
\end{align} 
where $c$ is a positive universal constant.

Talagrand’s theorem has become a cornerstone of combinatorics and probability theory.  
Its applications include the analysis of noise sensitivity of Boolean functions \cite{BKS1999}, the study of sharp threshold phenomena in random graphs \cite{FKKK2018}, and the investigation of geometric influences in product spaces \cite{KK2013,T1997}.  
Explicit examples demonstrating the tightness of Talagrand’s bound can be found in \cite{KKM2016,T1996}. An alternative quantitative correlation inequality is given in \cite{KMS2014}. For simplicity, we do not discuss this latter result here.

A family \(\mathcal{A}\subset \Omega_n\) is \emph{regular} when all its influences coincide.  
For increasing and regular families \(\mathcal{A},\mathcal{B}\subset \Omega_n\), Kalai, Keller and Mossel \cite{KKM2016} prove 
\begin{align}\label{Eq_KKM_result}
\mathrm{Cov}_\mu(\mathbf{1}_\mathcal{A},\mathbf{1}_\mathcal{B})
\geq c\,\frac{\mathcal{W}_1(\mathcal{A},\mathcal{B})}{\sqrt{\log\frac{e}{\mathcal{W}_1(\mathcal{A},\mathcal{A})}}\sqrt{\log\frac{e}{\mathcal{W}_1(\mathcal{B},\mathcal{B})}}}
=c\,\frac{I(\mathcal{A})I(\mathcal{B})}{n\sqrt{\log\frac{en}{I(\mathcal{A})^2}}\sqrt{\log\frac{en}{I(\mathcal{B})^2}}},
\end{align}
where \(c>0\) is universal.   It is easy to show that \eqref{Eq_KKM_result} is always at least as strong as \eqref{Eq_Talagrand_result}, see  \cite[Claim~3.1]{KKM2016} for details.
At the same time, their argument could apply to more general families as well, and we refer the reader to \cite{KKM2016} for the full statements.

In certain special cases, the lower bound \eqref{Eq_KKM_result} markedly improves upon \eqref{Eq_Talagrand_result}.  Consider an increasing, regular and balanced (i.e. $\mu(\mathcal{A})=1/2$) family $\mathcal{A}\subset \Omega_n$ and  $\mathcal{B}=\{x\in \Omega_n: \sum_{i=1}^{n}x_i> n/2\}$.  Here Talagrand’s bound is $\Omega(1/\sqrt{n})$, whereas the KKM-type bound is $\Omega(\sqrt{\log n}/\sqrt{n})$, i.e., sharper by a factor of $\sqrt{\log n}$; details are given in Appendix \ref{apendix_supplement}.

This example is central: $\mathbf{1}_\mathcal{B}$ is the standard majority function, a cornerstone of voting, threshold phenomena and complexity theory \cite{O2014}.  Nevertheless, $\Omega(\sqrt{\log n}/\sqrt{n})$ may still be sub-optimal.  Kalai et al. \cite{KKM2016} stated that, when $\mathcal{A}$ is the \emph{tribes} family, i.e., the balanced increasing function obtained by partitioning $[n]$ into disjoint blocks (''tribes'') of size $r\approx \log n-\log\log n$ and declaring $x\in\mathcal A$ iff at least one tribe is all-ones, then
\[
\mathrm{Cov}_\mu(\mathbf{1}_\mathcal{A}, \mathbf{1}_{\mathcal{B}})=\Theta(\log n/\sqrt{n}).
\]
Moreover, they conjectured that the same logarithmic advantage should hold in full generality.
In this paper our first result confirms that this conjecture is indeed attainable. Before presenting our result, we first set up the required notation.

A Boolean function is a mapping  
\(f : \Omega_n \to \{0,1\}\)  
that assigns a single binary value to each \(n\)-bit string.  
A \emph{linear threshold function } (LTF) is a Boolean function defined by  
\(f(x) = \mathbf{1}_{\{\sum_{i=1}^n a_i x_i > t\}}\),  
where \(a_1,\dots,a_n\in\mathbb R\) are weights and \(t\in\mathbb R\) is a threshold.
A Boolean function \(f\) is \emph{monotone} if  
\(x\le y\) coordinate-wise implies \(f(x)\le f(y)\).
The \emph{majority function}, denoted \(\mathrm{Maj}(x)\), is the monotone LTF, i.e.  
\(\mathrm{Maj}(x)=\mathbf{1}_{\{\sum_{i=1}^n x_i>n/2\}}\). We write $\mathbf{x}_{i}(x)=x_i$ ($1\le i\le n$) for the $i$-th coordinate projection.

With this notation in place, we can state our first main result.
\begin{theorem}\label{Theo_buchong}
    Let $\mathcal{A}\subset \Omega_n$ be an increasing, balanced, and regular family. Then there exists an absolute constant $c>0$ such that
    \begin{align}
        \mathrm{Cov}_\mu (\mathbf{1}_{\mathcal{A}}, \mathrm{Maj})\ge c\frac{\log n}{\sqrt{n}}.
    \end{align}
\end{theorem}

Theorem~\ref{Theo_buchong} resolves the regular case of the conjecture and provides an improved bound compared to Corollary 1.10 in \cite{KKM2016}. In addition, Kalai et al. \cite{KKM2016} proposed the following conjecture:
\begin{conjecture}[Conjecture 1.11, \cite{KKM2016}]\label{Con_1}
For every increasing and balanced family $\mathcal{A}\subseteq\Omega_n$ there exists an increasing linear-threshold family  
$\mathcal{B}=\bigl\{x\in \Omega_n:\sum_{i=1}^n a_i x_i>t\bigr\}$ with non-negative weights $a_i$ such that
\[
\operatorname{Cov}_\mu(\mathbf{1}_\mathcal{A},\mathbf{1}_\mathcal{B})\geq c\,\frac{\log n}{\sqrt{n}},
\]
where $c>0$ is a universal constant.
\end{conjecture}

Our second result solve this conjecture.
\begin{theorem}\label{Theo_main1}
    For every increasing and balanced family $\mathcal{A}\subseteq\Omega_n$, there exists  a Boolean function $h\in\{\mathbf{x}_1, \cdots, \mathbf{x}_n, \mathrm{Maj}\}$ such that
\[
\operatorname{Cov}_\mu(\mathbf{1}_\mathcal{A}, h)\geq c\,\frac{\log n}{\sqrt{n}},
\]
where $c>0$ is a universal constant.
\end{theorem}
Note that for the coordinate projection \(f_i(x)=x_i\) (\(1\le i\le n\)), the corresponding family \(\mathcal{B}\) in Conjecture \ref{Con_1} is simply the dictator half-space \(\{x\in\Omega_n:x_i>0\}\); for the majority function \(f(x)=\mathrm{Maj}(x)\) we obtain \(\mathcal{B}=\{x\in\Omega_n:\sum_i x_i>n/2\}\).  
Thus Theorem \ref{Theo_main1} directly verifies Conjecture \ref{Con_1}.

The preceding discussion of correlation inequalities focuses on the uniform measure over $\Omega_n$. Equally important, however, are the corresponding statements under Gaussian measure, to which we now turn.

Let $\gamma$ denote the standard Gaussian measure $\mathcal N(0,I_n)$ on $\mathbb R^n$. In a celebrated work, Royen \cite{R2014} proved the long-standing conjecture that any two symmetric convex sets are non-negatively correlated under $\gamma$. For qualitative correlation inequalities, Keller et al.\ \cite{KMS2014} established a Gaussian version of Talagrand’s correlation inequality for increasing families, and De et al. \cite{DNS2022} later extended this result to symmetric convex families.

As discussed earlier, the KKM-type correlation inequality (see \eqref{Eq_KKM_result}) can yield tighter bounds than its Talagrand-type counterpart in certain settings.  
To the best of our knowledge, no Gaussian analogue has been recorded.  
\textbf{Problem 6.3} of \cite{KKM2016} asks whether such an analogue exists.  
Our second main result establishes a Gaussian KKM-type inequality for the class of LTF sets.

Consider two families \(\mathcal{A}, \mathcal{B}\subset \mathbb{R}^n\).  
Under the Gaussian measure \(\gamma\), define the influence of the \(k\)-th coordinate on \(\mathcal{A}\) by  
\[
I_k^{(\gamma)}(\mathcal{A})=\mathbb{E}_{\gamma}\!\bigl[\mathbf{1}_{\mathcal{A}}\,\mathbf{x}_k\bigr]
\]  
and set  
\[
\mathcal{W}_1(\mathcal{A}, \mathcal{B})=\sum_{k=1}^n I_k^{(\gamma)}(\mathcal{A})\,I_k^{(\gamma)}(\mathcal{B}),
\]  
where $\mathbf{x}_k$ is the $k$-th coordinate projection defined as above. The intuition behind this notion of Gaussian influence is sketched in Appendix \ref{appendix_2}.

\begin{theorem}\label{Theo_main_2}
    Let $w=(w_1, \cdots, w_n),v=(v_1, \cdots, v_n)$ be vectors in $\mathbb{R}^n$ such that $\Vert w\Vert_2=\Vert v\Vert_2=1$ and define $\rho=\langle w, v\rangle$. Given $t, s\in \mathbb{R}$, define the following LTF sets:
    \begin{align}
        \mathcal{A}=\{x\in\mathbb{R}^n: \sum_{i}x_i w_i>t \},\quad \mathcal{B}=\{x\in\mathbb{R}^n: \sum_{i}x_i v_i>s \}.
    \end{align}
    Then, we have
    \begin{align}\label{lower bound for gaussian}
        \mathrm{Cov}_\gamma(\mathbf{1}_{\mathcal{A}},\mathbf{1}_{\mathcal{B}})\ge  c\,\frac{\mathcal{W}^{(\gamma)}_1(\mathcal{A},\mathcal{B})}{\sqrt{\log\frac{e}{\mathcal{W}^{(\gamma)}_1(\mathcal{A},\mathcal{A})}}\sqrt{\log\frac{e}{\mathcal{W}^{(\gamma)}_1(\mathcal{B},\mathcal{B})}}},
    \end{align}
    where $c>0$ is a universal constant.
\end{theorem}

 Next, we give an example in which the bound in Theorem~\ref{Theo_main_2} is tight up to an absolute constant.

\begin{proposition}\label{Prop_1}
     Let $w=(w_1, \cdots, w_n)$ be a vector in $\mathbb{R}^n$ such that $\Vert w\Vert_2=1$. Given $t\ge 1$, define the following LTF sets:
    \begin{align}
        \mathcal{A}=\{x\in\mathbb{R}^n: \sum_{i}x_i w_i>t \},\quad \mathcal{B}=\{x\in\mathbb{R}^n: \sum_{i}x_i w_i>-t \}.
    \end{align}
    Then, we have
    \begin{align}      \mathrm{Cov}_\gamma(\mathbf{1}_{\mathcal{A}},\mathbf{1}_{\mathcal{B}})\le  \frac{2\mathcal{W}^{(\gamma)}_1(\mathcal{A},\mathcal{B})}{\sqrt{\log\frac{e}{\mathcal{W}^{(\gamma)}_1(\mathcal{A},\mathcal{A})}}\sqrt{\log\frac{e}{\mathcal{W}^{(\gamma)}_1(\mathcal{B},\mathcal{B})}}}.
    \end{align}
\end{proposition}

Next, we extend Theorem \ref{Theo_main_2} to a broader setting.
\begin{theorem}\label{Theo_3}
Let $w,v\in\mathbb R^{n}$ with $\|w\|_{2}=\|v\|_{2}=1$ and $\rho=\langle w,v\rangle\ge 0$.  
Let $f,g:\mathbb R\to[0,1]$ be non-decreasing and left-continuous, and set  for $x\in\mathbb{R}^n$
$$F(x)=f\!\bigl(\langle w,x\rangle\bigr),\qquad G(x)=g\!\bigl(\langle v,x\rangle\bigr).$$
Then, there exists a universal constant $c>0$ such that
\[ \mathrm{Cov}_{\gamma}(F, G) \geq c \frac{\mathcal{W}^{(\gamma)}_1(F, G)}{\sqrt{\log \frac{e}{\mathcal{W}^{(\gamma)}_1(F, F)}} \sqrt{\log \frac{e}{{\mathcal{W}}^{(\gamma)}_1(G, G)}}}, \]
where $\mathcal{W}^{(\gamma)}_1(F, G)=\sum_k \mathbb{E}_\gamma[F\cdot\mathbf{x}_k]\mathbb{E}_\gamma[G\cdot\mathbf{x}_k]$. 
\end{theorem}

\section{Proofs of Theorems \ref{Theo_buchong} and \ref{Theo_main1}}

  We first state an approximation property that serves as the principal technical ingredient of the proofs.

\begin{lemma}[Theorem 1.7 in \cite{OW2013}]\label{Lem_1}
For every $0<\varepsilon<1/2$ there exists a constant $0<\delta=\delta(\varepsilon)<1$ such that the following holds.
Let $f$ be a monotone Boolean function and let $\mu$ denote the uniform probability measure on $\Omega_{n}$. Then at least one of the statements below must hold:
\begin{enumerate}[\rm(i)]
\item $\max\bigl\{\mu(f(x)=0),\;\mu(f(x)=1)\bigr\}\;\ge\;1-\varepsilon$;
\item $\displaystyle\max_{1\le i\le n}\mu\bigl(f(x)=x_i\bigr)\;\ge\;\frac{1}{2}+\frac{1}{n^{\varepsilon}}$;
\item $\mu\bigl(f(x)=\mathrm{Maj}(x)\bigr)\;\ge\;\frac{1}{2}+\delta\frac{\log n}{\sqrt{n}}$.
\end{enumerate}
\end{lemma}
\begin{proof}[Proof of Theorem \ref{Theo_buchong}]
 By Lemma~\ref{Lem_1} applied with $\varepsilon=\frac{1}{3}$ and $f=\mathbf{1}_{\mathcal{A}}$, one of the three alternatives must hold.

 It is obvious that the alternative (i) in Lemma \ref{Lem_1} is impossible since $\mathcal{A}$ is balanced. We next rule out alternative (ii) in Lemma \ref{Lem_1}. 

 Due to that $\mathbf{1}_{\mathcal{A}}$ is increasing, we have for $i\in [n]$
 \begin{align}
     &\mu\left(\mathbf{1}_{\mathcal{A}}=1\vert x_i=1 \right)-\mu\left(\mathbf{1}_{\mathcal{A}}=1\vert x_i=0 \right) \nonumber\\
     =&2\left(\mu\left(\mathbf{1}_{\mathcal{A}}=1, x_i=1 \right)-\mu\left(\mathbf{1}_{\mathcal{A}}=1, x_i=0 \right)   \right)\nonumber\\
     =& 4I_i(\mathcal{A}) =:T\ge 0.\nonumber
 \end{align}
 Note that $T$ does not depend on the index $i$ since $\mathcal{A}$ is regular.

For each $i\in [n]$,
\begin{align}
    \mu\left(\mathbf{1}_{\mathcal{A}}(x)=x_i \right)&=\frac{1}{2}\left(\mu\left(\mathbf{1}_{\mathcal{A}}=1\vert x_i=1 \right)+\mu\left(\mathbf{1}_{\mathcal{A}}=0\vert x_i=0 \right) \right)\nonumber\\
    &=\frac{1}{2}+\frac{1}{2}\left(\mu\left(\mathbf{1}_{\mathcal{A}}=1\vert x_i=1 \right)-\mu\left(\mathbf{1}_{\mathcal{A}}=1\vert x_i=0 \right)  \right)\nonumber\\
    &=\frac{1}{2}+2T.
\end{align}
 For $x\in \Omega_n$ and $i\in [n]$, define
 \begin{align}
     g(x)=2\mathbf{1}_{\mathcal{A}}(x)-1,\quad \chi_i(x)=2x_i-1.\nonumber
 \end{align}
 Obviously, the functions $\chi_1,\cdots,\chi_n$ are orthonormal in $L^2(\Omega_n, \mu)$. Bessel's inequality yields that
 \begin{align}\label{Eq_Bessel}
     \sum_{i=1}^n \bigl(\mathbb E_\mu[g(x)\chi_i(x)]\bigr)^2 \le \mathbb E_\mu[g(x)^2]=1
 \end{align}
 Note that for each $i$,
 \begin{align}
     \mathbb{E}_{\mu}[g(x)\chi_i(x)]&=\mu(\mathbf{1}_{\mathcal{A}}(x)=1\vert x_i=1)-\mu(\mathbf{1}_{\mathcal{A}}(x)=1\vert x_i=0)=T.
 \end{align}
 Therefore, we have by \eqref{Eq_Bessel}
 \begin{align}
     T\le \frac{1}{\sqrt{n}}.\nonumber
 \end{align}
Hence, we have 
\begin{align}
    \mu(\mathbf{1}_{\mathcal{A}}(x)=x_i)\le \frac{1}{2}+\frac{2}{\sqrt{n}}\le \frac{1}{2}+\frac{1}{n^{1/3}},
\end{align}
implying the alternative (ii) in Lemma \ref{Lem_1} is impossible. Hence, there exists a constant $\delta$ such that
\begin{align}
    \mu(\mathbf{1}_{\mathcal{A}}(x)=\mathrm{Maj}(x))\ge \frac{1}{2}+\frac{\delta}{2}\frac{\log n}{\sqrt{n}}.
\end{align}
Let $\mathcal{B}=\{x\in \Omega_n: \sum_i x_i>\frac{n}{2} \}$. Then $\mathrm{Maj}(x)=\mathbf{1}_{\mathcal{B}}(x)$. We conclude the proof by noting that
\begin{align}
    \mu(\mathbf{1}_{\mathcal{A}}=\mathbf{1}_{\mathcal{B}})&=\mathbb{E}_{\mu}[\mathbf{1}_{\mathcal{A}}\mathbf{1}_{\mathcal{B}}]-\mathbb{E}_{\mu}[(1-\mathbf{1}_{\mathcal{A}})(1-\mathbf{1}_{\mathcal{B}})]\nonumber\\
    &=1-\mathbb{E}_\mu[\mathbf{1}_{\mathcal{A}}]-\mathbb{E}_\mu[\mathbf{1}_{\mathcal{B}}]+2\mathbb{E}_\mu[\mathbf{1}_{\mathcal{A}}\mathbf{1}_{\mathcal{B}}]\nonumber\\
    &=\frac{1}{2}+2\left(\mathbb{E}_\mu[\mathbf{1}_{\mathcal{A}}\mathbf{1}_{\mathcal{B}}]-\mathbb{E}_\mu[\mathbf{1}_{\mathcal{A}}]\mathbb{E}_\mu[\mathbf{1}_{\mathcal{B}}] \right)\nonumber\\
    &=\frac{1}{2}+2\mathrm{Cov}_\mu(\mathbf{1}_{\mathcal{A}}, \mathbf{1}_{\mathcal{B}}),
\end{align}
where the third inequality is due to that $\mathcal{A}$ is balanced.
\end{proof}

\begin{proof}[Proof of Theorem \ref{Theo_main1}]
    Note that, for any Boolean function $h$, we have
    \begin{align}
        \mathrm{Cov}_\mu(\mathbf{1}_{\mathcal{A}}, h)&=\mathbb{E}_\mu(\mathbf{1}_{\mathcal{A}}\cdot h)-\mathbb{E}_\mu\mathbf{1}_{\mathcal{A}}\cdot\mathbb{E}_\mu h\nonumber\\&=\mu(\mathbf{1}_{\mathcal{A}}=h=1)-\mu(\mathbf{1}_{\mathcal{A}}=1)\cdot\mu(h=1)\nonumber\\
        &=\mu(\mathbf{1}_{\mathcal{A}}=h=1)-\frac{1}{2}\cdot\mu(h=1).
    \end{align}
    Since
    \begin{align}
\mu(\mathbf{1}_{\mathcal{A}}=h)&=\mu(\mathbf{1}_{\mathcal{A}}=h=1)+\mu(\mathbf{1}_{\mathcal{A}}=h=0)\nonumber\\
&=2\mu(\mathbf{1}_{\mathcal{A}}=h=1)-\mu(h=1)+\mu(\mathbf{\mathbf{1}_{\mathcal{A}}}=0),
    \end{align}
    and $\mathcal{A}$ is balanced, we have
    \begin{align}
        \mu(\mathbf{1}_{\mathcal{A}}=h=1)=\frac{\mu(\mathbf{1}_{\mathcal{A}}=h)+\mu(h=1)-1/2}{2}.
    \end{align}
    Hence, we have 
    \begin{align}\label{Eq_proof_1}
 \mathrm{Cov}_\mu(\mathbf{1}_{\mathcal{A}}, h)=\frac{1}{2}\mu(\mathbf{1}_{\mathcal{A}}=h)-\frac{1}{4}.
    \end{align}

    Since \(\mathcal{A}\) is increasing, its indicator \(\mathbf{1}_{\mathcal{A}}\) is a monotone Boolean function.  
Apply Lemma~\ref{Lem_1} to \(\mathbf{1}_{\mathcal{A}}\) with \(\varepsilon=1/4\).  There exists a universal constant \(0<\delta<1\) such that at least one of the following holds:

\begin{enumerate}[\rm(i)]
\item \(\max\bigl\{\mu(\mathbf{1}_{\mathcal{A}}=0),\;\mu(\mathbf{1}_{\mathcal{A}}=1)\bigr\}\ge\frac{3}{4}\);
\item \(\displaystyle\max_{1\le i\le n}\mu\bigl(\mathbf{1}_{\mathcal{A}}=x_i\bigr)\ge\frac{1}{2}+\frac{1}{n^{1/4}}\);
\item \(\mu\bigl(\mathbf{1}_{\mathcal{A}}=\mathrm{Maj}\bigr)\ge\frac{1}{2}+\delta\frac{\log n}{\sqrt{n}}\).
\end{enumerate}
Note that  
\[
\mu(\mathbf{1}_{\mathcal{A}}=1)=\mu(\mathcal{A})=\frac12<\frac34,
\]  
and the same bound holds for \(\mu(\mathbf{1}_{\mathcal{A}}=0)\).  
Therefore, statement (i) cannot hold. Note that, there exists a universal constant $c>0$ such that
\begin{align}
    \min\{\frac{1}{n^{1/4}},  \delta\frac{\log n}{\sqrt{n}} \}\ge c\frac{\log n}{\sqrt{n}}.
\end{align}
Hence, there exists a Boolean function $h\in \{\mathbf{x}_{1}, \cdots, \mathbf{x}_n, \mathrm{Maj}\}$ such that
\begin{align}
    \mu\big( \mathbf{1}_{\mathcal{A}}=h  \big)\ge \frac{1}{2}+c\frac{\log n}{\sqrt{n}}.
\end{align}
By virtue of \eqref{Eq_proof_1}, we have 
\begin{align}
    \mathrm{Cov}_\mu(\mathbf{1}_{\mathcal{A}}, h)\ge \frac{1}{2}\big( \frac{1}{2}+c\frac{\log n}{\sqrt{n}}  \big)-\frac{1}{4}=\frac{c\log n}{2\sqrt{n}},
\end{align}
which concludes the proof.

\end{proof}
\section{Proof of Theorem \ref{Theo_main_2}}

Before we begin the proof, we first introduce the sign function  
\[
\operatorname{sgn}(x)=
\begin{cases}
+1 & \text{if } x>0,\\[2pt]
-1 & \text{if } x\le0.
\end{cases}
\]
For clarity, we split the full proof into the following  lemmas.
\begin{lemma}\label{Lem_proof2_1}
   Let \(\xi=(\xi_1, \cdots, \xi_n)^\top\sim N(0,I_{n})\) and fix a unit vector \(w=(w_1, \cdots, w_n)^\top\in\mathbb{R}^{n}\) (\(\|w\|_{2}=1\)) and \(t\in\mathbb{R}\).  
For every \(k\in\{1,\dots,n\}\),

\[
\mathbb{E}\!\left[\operatorname{sgn}\!\Bigl(\sum_{i=1}^{n}w_{i}\xi_{i}-t\Bigr)\xi_{k}\right]=2\varphi(t)\,w_{k},\qquad \varphi(t)=\frac{1}{\sqrt{2\pi}}\,e^{-t^{2}/2}.
\]
\end{lemma}
\begin{proof}
    Let \(Y=\sum_{j=1}^{n}w_{j}\xi_{j}\) with \(\xi_{j}\sim N(0,1)\) i.i.d. and \(\|w\|_{2}=1\).  
Then the joint distribution of \((Y,\xi_{k})\) is

\[
(Y,\xi_{k})\sim N\!\left(\begin{pmatrix}0\\0\end{pmatrix},\begin{pmatrix}1&w_{k}\\w_{k}&1\end{pmatrix}\right).
\]

Conditioning on \(Y=y\) gives

\[
\xi_{k}\mid Y=y\;=\;w_{k}\,y\;+\;G,\qquad G\sim N(0,1-w_{k}^{2})\;\text{independent of }Y.
\]
Hence, we have by the conditional expectation formula
\[
\mathbb{E}\!\left[\operatorname{sgn}(Y-t)\xi_{k}\right]
=\mathbb{E}\!\Bigl[\operatorname{sgn}(Y-t)\,\mathbb{E}[\xi_{k}\mid Y]\Bigr]
=w_{k}\,\mathbb{E}\!\left[Y\operatorname{sgn}(Y-t)\right].
\]
The remaining expectation is
\[
\mathbb{E}\!\left[Y\cdot\operatorname{sgn}(Y-t)\right]
=\int_{-\infty}^{\infty}y\operatorname{sgn}(y-t)\frac{e^{-y^{2}/2}}{\sqrt{2\pi}}\,dy
=2\int_{t}^{\infty}y\frac{e^{-y^{2}/2}}{\sqrt{2\pi}}\,dy
=\frac{2}{\sqrt{2\pi}}e^{-t^{2}/2},
\]
which concludes the proof.
\end{proof}

\begin{lemma}\label{Lem_proof2_2}
    Let \(\xi\sim N(0,I_{n})\) and fix unit vectors \(w,v\in\mathbb{R}^{n}\) together with thresholds \(t,s\in\mathbb{R}\).  Assume that \(\rho=w^{\top}v>0\).
Then, there exists a universal constant \(c>0\) such that
\begin{equation}\label{sign form}
    \mathrm{Cov}\!\Bigl(\operatorname{sgn}(w^{\top}\xi-t),\,\operatorname{sgn}(v^{\top}\xi-s)\Bigr)
\ge c\,\frac{\rho\,\varphi(t)\varphi(s)}{(1+|t|)(1+|s|)},
\end{equation}
where \(\varphi(t)=\dfrac{1}{\sqrt{2\pi}}e^{-t^{2}/2}\).
\end{lemma}
\begin{proof}
	Note that, 
\begin{align}
    \mathrm{Cov}(w^\top\xi, v^\top\xi)=w^\top v=\rho.
\end{align}
    Hence, by Appendix \ref{appendix_3} and the Plackett formula (see \cite{P1954}),
\[
\begin{aligned}
&\operatorname{Cov}\!\Bigl(\operatorname{sgn}(w^{\top}\xi-t),\,\operatorname{sgn}(v^{\top}\xi-s)\Bigr)\\[2pt]
&\quad =4\Bigl[\mathbb{P}\bigl(w^\top \xi>t,\, v^\top\xi>s\bigr)
           -\mathbb{P}\bigl(w^\top \xi>t\bigr)\mathbb{P}\bigl(v^\top \xi>s\bigr)\Bigr]\\[4pt]
&\quad =4\int_{0}^{\rho}\phi_{r}(t,s)\,dr,
\end{aligned}
\]
where
\[
\phi_{r}(t,s)
=\frac{1}{2\pi\sqrt{1-r^{2}}}
\exp\!\biggl(-\frac{t^{2}+s^{2}-2rts}{2(1-r^{2})}\biggr).
\]
Define the following normalized ratio
\begin{align}
    \Gamma(t, s, \rho) := \frac{\text{Cov}(\operatorname{sgn}(w^{\top}\xi-t),\,\operatorname{sgn}(v^{\top}\xi-s))(1 + |t|)(1 + |s|)}{\rho \varphi(t) \varphi(s)} = 4 \frac{(1 + |t|)(1 + |s|)}{\rho} \int_0^\rho h_r(t, s) \, dr,
\end{align}
where
\begin{equation}\label{Eq_lem2_4}
h_r(t, s) = \frac{1}{\sqrt{1 - r^2}} \exp\left(\frac{rts - \frac{1}{2}(t^2 + s^2)r^2}{1 - r^2}\right).  
\end{equation}
We conclude the proof by Lemma \ref{Lem_3}.
\end{proof}

\begin{lemma}\label{Lem_3}
    Let $\Gamma(t, s, \rho)$ be the normalized ratio defined in the proof of Lemma \ref{Lem_proof2_2}. Then, we have for some absolute constant $c>0$
\begin{align}\label{Eq_lem2_1}
\Gamma(t,s,\rho)\ge c\qquad
\forall\,t,s\in\mathbb R,\;\rho\in(0,1].
\end{align}  
\end{lemma}
\begin{proof}
    Recall that
    \begin{align}
    \Gamma(t, s, \rho)  = 4 \frac{(1 + |t|)(1 + |s|)}{\rho} \int_0^\rho h_r(t, s) \, dr,
\end{align}
where $h_r(t, s)$ was defined in \eqref{Eq_lem2_4}.

We split the parameter space into three overlapping but exhaustive regions and prove a positive lower bound in each.

\textbf{Case I: $|t|\le 1$, $|s|\le 1$.}

On the compact set $(t,s,\rho)\in[-1,1]^{2}\times[0,1]$ the integrand $h_{r}(t,s)$ is continuous and strictly positive for $r>0$, while  
\[
\lim_{\rho\to 0^{+}}\Gamma(t,s,\rho)=4(1+|t|)(1+|s|)\ge 4>0.
\]
Hence $\Gamma$ extends continuously to $\rho=0$ and attains a positive minimum $c_{\mathrm I}>0$ on this square.

\textbf{Case II: $ts\ge 0, (t, s)\in\mathbb{R}^2$.}

Assume $t,s\ge 0$ (other sign combination is symmetric).  Note that for $|r| \leq 1/2$,
\[ \frac{rts - \frac{1}{2}(t^2 + s^2)r^2}{1 - r^2} \geq -\frac{\frac{1}{2}(t^2 + s^2)r^2}{1 - r^2}  \geq -2(t^2 + s^2)r^2. \]
Hence we have for $0\le r\le 1/2$
\[
h_{r}(t,s)\ge\exp\!\bigl(-2(t^{2}+s^{2})r^{2}\bigr).
\]
Set  
\[
r_{0}:=\min\!\Bigl\{\rho,\tfrac{1}{2\sqrt{t^{2}+s^{2}}}\Bigr\}\le\tfrac12.
\]
Then $(t^{2}+s^{2})r^{2}\le 1/4$ for $r\in[0,r_{0}]$, so  
\[
\int_{0}^{\rho}h_{r}\,dr\ge\int_{0}^{r_{0}}e^{-2(t^{2}+s^{2})r^{2}}dr\ge e^{-1/2}r_{0}.
\]
Plugging into $\Gamma$,
\[
\Gamma(t,s,\rho)\ge 4e^{-1/2}(1+t)(1+s)\cdot\frac{r_{0}}{\rho}.
\]

 If $\rho\le\tfrac{1}{2\sqrt{t^{2}+s^{2}}}$, then $r_{0}=\rho$ and $\Gamma\ge 4e^{-1/2}$.
 Otherwise $r_{0}=\tfrac{1}{2\sqrt{t^{2}+s^{2}}}$ and  
\[
\frac{(1+t)(1+s)}{\sqrt{t^{2}+s^{2}}}\ge\frac12 \quad\Longrightarrow\quad \Gamma\ge 2e^{-1/2}.
\]
Thus in \textbf{Case II}, $\Gamma\ge c_{\mathrm{II}}:=e^{-1/2}$.

\textbf{Case III: $ts<0, (t, s)\in\mathbb{R}^2$.}

Write $t>0$, $s=-k<0$ with $k>0$. By symmetry and \textbf{Case I}, assume $t,k\ge 1$.  
Lemma~\ref{Lem_appendix_4} gives $h_{r}(t,-k)\ge e^{-1}$ for $r\in[0,\tfrac{1}{2tk}]$.  

Set $r_{1}:=\min\!\bigl\{\rho,\tfrac{1}{2tk}\bigr\}$, then  
\[
\int_{0}^{\rho}h_{r}\,dr\ge e^{-1}r_{1},
\]
so
\[
\Gamma(t,-k,\rho)\ge 4e^{-1}(1+t)(1+k)\cdot\frac{r_{1}}{\rho}.
\]

If $\rho\le\tfrac{1}{2tk}$, then $r_{1}=\rho$ and $\Gamma\ge 4e^{-1}$.
If $\rho>\tfrac{1}{2tk}$, then $r_{1}=\tfrac{1}{2tk}$ and  
\[
\frac{(1+t)(1+k)}{tk}\ge 1 \quad\Longrightarrow\quad \Gamma\ge 2e^{-1}.
\]
Hence in Case III, $\Gamma\ge c_{\mathrm{III}}:=2e^{-1}$.

When $t < 1 \leq k$, let
\[
r_1 := \min \left\{ \rho, \frac{1}{2k} \right\}.
\]
Then $r_1 \leq 1/2$, and for $0 \leq r \leq r_1$:
\[
r t k \leq \frac{t}{2} \leq \frac{1}{2}, \quad (t^2 + k^2) r^2 \leq \frac{t^2 + k^2}{4k^2} \leq \frac{1 + k^2}{4k^2} \leq \frac{1}{2}.
\]

A similiar argument as Lemma~\ref{Lem_appendix_4} gives a constant lower bound for $\Gamma$.

Taking  
\[
c:=\min\{c_{\mathrm I},\,c_{\mathrm{II}},\,c_{\mathrm{III}}\}>0
\]
completes the proof.
\end{proof}

Now, we are prepared to prove the following result.
\begin{proof}[Proof of Theorem \ref{Theo_main_2}]

Set
\[
\tilde{\mathbf{1}}_{\mathcal{A}}(x)=
\begin{cases}
\phantom{-}1, & x\in\mathcal{A},\\[2pt]
-1, & x\notin\mathcal{A},
\end{cases}
\]  
and define the signed influence of the $k$-th coordinate by  
\[
\tilde{I}_k^{(\gamma)}(\mathcal{A})=\mathbb{E}_{\gamma}\!\bigl[\tilde{\mathbf{1}}_{\mathcal{A}}\,\mathbf{x}_k\bigr].
\]  
Put  
\[
\tilde{\mathcal{W}}_1(\mathcal{A},\mathcal{B})=\sum_{k=1}^{n}\tilde{I}_k^{(\gamma)}(\mathcal{A})\,\tilde{I}_k^{(\gamma)}(\mathcal{B}).
\]  
Since $\mathbf{1}_{\mathcal{A}}=\tfrac12\bigl(\tilde{\mathbf{1}}_{\mathcal{A}}+1\bigr)$, Theorem~\ref{Theo_main_2} follows once we establish  
\[
\mathrm{Cov}_{\gamma}\!\bigl(\tilde{\mathbf{1}}_{\mathcal{A}},\tilde{\mathbf{1}}_{\mathcal{B}}\bigr)
\ge c\,\frac{\tilde{\mathcal{W}}^{(\gamma)}_1(\mathcal{A},\mathcal{B})}
{\sqrt{\log\frac{e}{\tilde{\mathcal{W}}^{(\gamma)}_1(\mathcal{A},\mathcal{A})}}
 \sqrt{\log\frac{e}{\tilde{\mathcal{W}}^{(\gamma)}_1(\mathcal{B},\mathcal{B})}}}.
\]

    Lemma \ref{Lem_proof2_1} yields that
    \begin{align}
        \tilde{\mathcal{W}}^{(\gamma)}_1(\mathcal{A}, \mathcal{B})=4\varphi(t)\varphi(s)\sum_{k=1}^nw_k v_k=4\rho \varphi(t)\varphi(s),
    \end{align}
    and
    \begin{align}
        \tilde{\mathcal{W}}^{(\gamma)}_1(\mathcal{A}, \mathcal{A})=4\varphi^2(t),\quad
        \tilde{\mathcal{W}}^{(\gamma)}_1(\mathcal{B}, \mathcal{B})=4\varphi^2(s).
    \end{align}
    Hence, we have 
    \begin{align}
        \log\frac{e}{\tilde{\mathcal{W}}^{(\gamma)}_1(\mathcal{A},\mathcal{A})}=\log\frac{e}{4\varphi^2(t)}=t^2+\log(\frac{e\pi}{2}),
    \end{align}
    which yields that, for some universal constant $c_1>0$
    \begin{align}
       \frac{1}{c_1}(1+\vert t\vert)\le  \sqrt{\log\frac{e}{\tilde{\mathcal{W}}^{(\gamma)}_1(\mathcal{A},\mathcal{A})}}\le c_1(1+\vert t\vert).
    \end{align}
    Similarly, there exists a universal constant $c_2>0$ such that
    \begin{align}
       \frac{1}{c_2}(1+\vert s\vert)\le  \sqrt{\log\frac{e}{\tilde{\mathcal{W}}^{(\gamma)}_1(\mathcal{B},\mathcal{B})}}\le c_2(1+\vert s\vert).
    \end{align}

   Applying Lemma \ref{Lem_proof2_2}, we have for some universal constant $c_3>0$
    \begin{align}
        \mathrm{Cov}_{\gamma}\!\bigl(\tilde{\mathbf{1}}_{\mathcal{A}},\tilde{\mathbf{1}}_{\mathcal{B}}\bigr)\ge c_3\frac{\rho\varphi(t)\varphi(s)}{(1+\vert t\vert)(1+\vert s\vert)},
    \end{align}
which concludes the proof.
\end{proof}

\section{Proof of Proposition \ref{Prop_1}}
\begin{proof}[Proof of Proposition \ref{Prop_1}]
Let $\Phi(\cdot)$ and $\varphi(\cdot)$ be the cumulative distribution function and probability density function, of the standard normal distribution.  
Because $\mathcal A\subseteq\mathcal B$,  
\[
\operatorname{Cov}_\gamma(\mathbf1_{\mathcal A},\mathbf1_{\mathcal B})
=\gamma(\mathcal A)\bigl(1-\gamma(\mathcal B)\bigr)
=\Phi(-t)^2.
\]
The standard Gaussian tail bound $\Phi(-t)\le\varphi(t)/t$ ($t>0$) gives  

\[
\operatorname{Cov}_\gamma(\mathbf1_{\mathcal A},\mathbf1_{\mathcal B})\le\frac{\varphi(t)^2}{t^2}.
\]

Lemma~\ref{Lem_proof2_1} yields for $k\in\{1, \cdots, n\}$  
\[
I_k^{(\gamma)}(\mathcal A)=w_k\varphi(t),\qquad
I_k^{(\gamma)}(\mathcal B)=w_k\varphi(-t)=w_k\varphi(t),
\]
so that  
\[
\mathcal W_1^{(\gamma)}(\mathcal A,\mathcal B)=\varphi(t)^2,\qquad
\mathcal W_1^{(\gamma)}(\mathcal A,\mathcal A)=\mathcal W_1^{(\gamma)}(\mathcal B,\mathcal B)=\varphi(t)^2.
\]
Consequently,
\[
\frac{\mathcal W_1^{(\gamma)}(\mathcal A,\mathcal B)}
{\sqrt{\log\dfrac{e}{\mathcal W_1^{(\gamma)}(\mathcal A,\mathcal A)}}\;
\sqrt{\log\dfrac{e}{\mathcal W_1^{(\gamma)}(\mathcal B,\mathcal B)}}}
=\frac{\varphi(t)^2}{\log\!\bigl(e/\varphi(t)^2\bigr)}
=\frac{\varphi(t)^2}{t^2+1}.
\]

For $t\ge 1$ we have $t^2/(t^2+1)\ge\frac12$, completing the proof.
\end{proof}

\section{Proof of Theorem \ref{Theo_3}}
In this section we fix vectors \(w,v\in\mathbb R^{n}\) with \(\|w\|_{2}=\|v\|_{2}=1\) and \(\rho=\langle w,v\rangle\ge 0\).  
For non-decreasing left-continuous functions \(f,g:\mathbb R\to[0,1]\) define  
\[
F(x)=f(\langle w,x\rangle),\qquad G(x)=g(\langle v,x\rangle)\qquad(x\in\mathbb R^{n}).
\]  
The corresponding Lebesgue–Stieltjes measures are given by  
\[
\mu_f([a,b))=f(b)-f(a),\qquad \mu_g([a,b))=g(b)-g(a).
\]
\begin{lemma}\label{Lem_Theorem3_1}
    Let $X=(X_1, \cdots, X_n)\sim N(0, I_n)$. Then we have
    \begin{align}
        \sum_k\mathbb{E}[F(X)\cdot X_k]\mathbb{E}[G(X)\cdot X_k]=\rho \mathbb{E}[f(X_1)X_1]\mathbb{E}[g(X_1)X_1]
    \end{align}
    and
    \begin{align}
        \sum_k\big(\mathbb{E}[F(X)\cdot X_k]\big)^2=\big(\mathbb{E}[f(X_1)X_1]\big)^2,\quad \sum_k\big(\mathbb{E}[G(X)\cdot X_k]\big)^2=\big(\mathbb{E}[g(X_1)X_1]\big)^2.
    \end{align}
\end{lemma}
\begin{proof}
    Note that
    \begin{align}
        \mathbb{E}[F(X)X]=\mathbb{E}[f(\langle w, X\rangle)(\langle w, X\rangle\cdot w+X-\langle w, X\rangle\cdot w)].
    \end{align}
    It is easy to verify that $X-\langle w, X\rangle\cdot w$ is independent of $\langle w, X\rangle$. Hence, we have
     \begin{align}\label{Eq_Lemma5.1_1}
        \mathbb{E}[F(X)X]=\mathbb{E}[f(\langle w, X\rangle)\langle w, X\rangle\cdot w]=\mathbb{E}[f(X_1)X_1]\cdot w.
    \end{align}
    Then, we can conclude the proof with \eqref{Eq_Lemma5.1_1}.
\end{proof}

\begin{lemma}\label{Lem_theo3_2}
Let
\[
\varphi(t)=\frac{1}{\sqrt{2\pi}}e^{-t^2/2},
\qquad
a_f:=\int_{\mathbb R}\varphi(t)\,d\mu_f(t), \qquad m_f:= \mu_f(\mathbb R)
\]
If $a_f=0$, then $\mu_f=0$ and hence
\[
\int_{\mathbb R}(1+|t|)\varphi(t)\,d\mu_f(t)=0.
\]
If $a_f>0$, then
\[
\int_{\mathbb R}(1+|t|)\varphi(t)\,d\mu_f(t)
\le
a_f\Bigl(1+\sqrt{2\log\frac{m_f}{\sqrt{2\pi}\,a_f}}\Bigr)
\le
a_f\Bigl(1+\sqrt{\log\frac{1}{2\pi a_f^2}}\Bigr).
\]
In particular, there exists a universal constant $C>0$ such that
\[
\int_{\mathbb R}(1+|t|)\varphi(t)\,d\mu_f(t)
\le
C\,a_f\sqrt{\log\frac{e}{a_f^2}}.
\]
\end{lemma}

\begin{proof}
If $m_f=0$, then $\mu_f=0$, so the claim is trivial. Assume henceforth that $m_f>0$.
Let
\[
\nu_f:=\frac{\mu_f}{m_f},
\]
and let $T\sim \nu_f$. Set
\[
S:=|T|,
\qquad
U:=e^{-S^2/2}\in(0,1].
\]
Then
\[
a_f
=
\int_{\mathbb R}\varphi(t)\,d\mu_f(t)
=
\frac{m_f}{\sqrt{2\pi}}\mathbb E[U],
\]
and
\[
\int_{\mathbb R}(1+|t|)\varphi(t)\,d\mu_f(t)
=
\frac{m_f}{\sqrt{2\pi}}\mathbb E[(1+S)U].
\]
Since $S=\sqrt{-2\log U}$, if we define
\[
h(u):=u\bigl(1+\sqrt{-2\log u}\bigr),
\qquad 0<u\le 1,
\]
then
\[
\int_{\mathbb R}(1+|t|)\varphi(t)\,d\mu_f(t)
=
\frac{m_f}{\sqrt{2\pi}}\mathbb E[h(U)].
\]

A direct computation gives
\[
h''(u)
=
-\frac{1}{u\sqrt{-2\log u}}
-\frac{1}{u(-2\log u)^{3/2}}
<0
\qquad (0<u<1),
\]
so $h$ is concave on $(0,1]$. Therefore, by Jensen's inequality,
\[
\mathbb E[h(U)]
\le
h(\mathbb E[U])
=
\mathbb E[U]\Bigl(1+\sqrt{-2\log \mathbb E[U]}\Bigr).
\]
Multiplying by $m_f/\sqrt{2\pi}$ and using
\[
\mathbb E[U]=\frac{\sqrt{2\pi}\,a_f}{m_f},
\]
we obtain
\[
\int_{\mathbb R}(1+|t|)\varphi(t)\,d\mu_f(t)
\le
a_f\Bigl(1+\sqrt{-2\log\frac{\sqrt{2\pi}\,a_f}{m_f}}\Bigr)
=
a_f\Bigl(1+\sqrt{2\log\frac{m_f}{\sqrt{2\pi}\,a_f}}\Bigr).
\]
Since $m_f\le 1$, this implies
\[
\int_{\mathbb R}(1+|t|)\varphi(t)\,d\mu_f(t)
\le
a_f\Bigl(1+\sqrt{2\log\frac{1}{\sqrt{2\pi}\,a_f}}\Bigr)
=
a_f\Bigl(1+\sqrt{\log\frac{1}{2\pi a_f^2}}\Bigr).
\]

Finally, since
\[
0<a_f\le \frac{m_f}{\sqrt{2\pi}}\le \frac{1}{\sqrt{2\pi}},
\]
we have $\log(1/(2\pi a_f^2))\ge 0$, and hence
\[
1+\sqrt{\log\frac{1}{2\pi a_f^2}}
\le
2\sqrt{1+\log\frac{1}{2\pi a_f^2}}
=
2\sqrt{\log\frac{e}{2\pi a_f^2}}
\le
2\sqrt{\log\frac{e}{a_f^2}}.
\]
Thus the final estimate holds with, for instance, $C=2$.
\end{proof}

\begin{proof}[Proof of Theorem \ref{Theo_3}]
    Let $(Z_1, Z_2)$ be a centered normal pair with $\mathrm{Var}(Z_i)=1$ and  $\mathrm{Cov}(Z_1, Z_2) = \rho$. By the definitions of $\mu_f$ and $\mu_g$, we have
\[ f(x) = \int \textbf{1}_{\{x > t\}} \, d\mu_f(t), \quad g(y) = \int \textbf{1}_{\{y > s\}} \, d\mu_g(s). \]
     Then by Fubini's theorem, we have
\begin{equation}\label{CORO_2}
    \mathrm{Cov}_{\gamma}(F, G)=\mathrm{Cov}(f(Z_1), g(Z_2)) = \iint \mathrm{Cov}(\textbf{1}_{\{Z_1 > t\}}, \textbf{1}_{\{Z_2 > s\}}) \, d\mu_f(t) \, d\mu_g(s).
\end{equation} 
 Note that $\textbf{1}_{\{z > t\}} = \frac{1}{2}(1 + \mathrm{sgn}(z - t))$, hence Lemma \ref{Lem_proof2_2} yields
\[ \text{Cov}(1_{\{Z_1 > t\}}, 1_{\{Z_2 > s\}}) = \frac{1}{4} \text{Cov}(\text{sgn}(Z_1 - t), \text{sgn}(Z_2 - s)) \geq c_1 \cdot \frac{\rho \varphi(t) \varphi(s)}{(1 + |t|)(1 + |s|)}. \]
Substituting into \eqref{CORO_2}, we get
\[ \text{Cov}(f(Z_1), g(Z_2)) \geq c_1 \rho \left( \int \frac{\varphi(t)}{1 + |t|} d\mu_f(t) \right) \left( \int \frac{\varphi(s)}{1 + |s|} d\mu_g(s) \right)=:c_1\rho A_f A_g. \]

Recall
$$
a_f=\mathbb{E}[f(Z_1)Z_1]=\int \mathbb{E}[Z_1 \textbf{1}_{\{Z_1 > t\}}]d\mu_f(t) = \int \varphi(t) \, d\mu_f(t).
$$
Then by the Cauchy-Schwarz inequality, we have
$$
a_f^2 = \left( \int \varphi(t) \, d\mu_f(t) \right)^2 \leq \left( \int \frac{\varphi(t)}{1 + |t|} d\mu_f(t) \right) \left( \int (1 + |t|) \varphi(t) \, d\mu_f(t) \right) =: A_f B_f,
$$
So Lemma \ref{Lem_theo3_2} yields that
\[ A_f \geq \frac{a_f^2}{B_f}\ge \frac{1}{C} \frac{a_f}{\sqrt{\log(e/a_f^2)}}. \]
Similarly, we have
\[ A_g \geq \frac{a_g^2}{B_g}\ge \frac{a_g}{\sqrt{\log(e/a_g^2)}}, \]
where
\begin{align}
    a_g:=\mathbb{E}[g(Z_1)Z_1] = \int \varphi(t) \, d\mu_g(t).
\end{align}
We conclude the proof by Lemma \ref{Lem_Theorem3_1}.
\end{proof}

\appendix
\section{Supplementary Proofs for Corollaries of Talagrand’s and KKM’s Results } \label{apendix_supplement}

In this section we establish the special-case corollaries of the main Talagrand and KKM theorems that were announced in the Introduction.

Let $\mathcal{A}\subset \Omega_n$ be an increasing, regular and balanced family and  $\mathcal{B}=\{x\in \Omega_n: \sum_{i=1}^{n}x_i> n/2\}$. Recall that the \emph{influence} of the $k$-th coordinate of $\mathcal{A}$ is
\[
I_k(\mathcal{A})=2\mu\bigl(\{x\in\mathcal{A}:x\oplus e_k\notin\mathcal{A}\}\bigr),
\]
where $x\oplus e_k$ is obtained from $x$ by flipping the $k$-th coordinate.  
The \emph{total influence} of $\mathcal{A}$ is
\[
I(\mathcal{A})=\sum_{k=1}^n I_k(\mathcal{A}).
\]
Before giving our proof, we first quote the following lower bound on influences.

\begin{lemma}[Theorem 3.1 in \cite{KKL1988}]\label{Lem_appendix}
    For any family $\mathcal{A}\subset \Omega_n$, we have
    \begin{align}
        \max_{1\le i\le n}I_{i}(\mathcal{A})\ge c\frac{\log n}{n},
    \end{align}
    where $c>0$ is a universal constant.
\end{lemma}

Applying Lemma \ref{Lem_appendix} to the increasing, regular and balanced family $\mathcal{A}$, we have
\begin{align}
    I(\mathcal{A})\ge c\log n.
\end{align}
We next consider the influence of
\[
\mathcal{B}=\bigl\{x\in\Omega_{n}:\sum_{i=1}^{n}x_{i}>n/2\bigr\}.
\]
By symmetry, it suffices to compute $I_{1}(\mathcal{B})$. Let $x\in \Omega_n$ such that 
\begin{align}
    \sum_{i=1}^nx_i>\frac{n}{2},\quad (1-x_1)+\sum_{i=2}^{n}x_{i}\le \frac{n}{2}.
\end{align}
Hence, we have $x_1=1$. Then, we have
\begin{align}
    \frac{n}{2}-1<\sum_{i=2}^{n}x_i\le \frac{n}{2},
\end{align}
which yields that 
\begin{align}
    I_1(\mathcal{B})&=2\mu\big(x\in\Omega_n: x_1=1,  \frac{n}{2}-1<\sum_{i=2}^{n}x_i\le \frac{n}{2} \big)\nonumber\\
    &=2\frac{\binom{n-1}{\lfloor n/2\rfloor}}{2^n}=c_1\frac{1}{\sqrt{n}},
\end{align}
where $c_1>0$ is a universal constant and the last equality is due to the Stirling formula. 

For the families $\mathcal{A}, \mathcal{B}$ introduced in this section, Talagrand's result yields that
\begin{align}
    \mathrm{Cov}_\mu(\mathbf{1}_{\mathcal{A}}, \mathbf{1}_{\mathcal{B}})\ge c_2\frac{I(\mathcal{A})I(\mathcal{B})}{n\log (en/I(\mathcal{A})I(\mathcal{B}))}\ge \frac{c_3}{n},
\end{align}
and KKM's result yields that
\begin{align}
    \mathrm{Cov}_\mu(\mathbf{1}_{\mathcal{A}}, \mathbf{1}_{\mathcal{B}})\ge c_4\frac{I(\mathcal{A})I(\mathcal{B})}{n\sqrt{\log (en/I^2(\mathcal{A}))}\sqrt{\log (en/I^2(\mathcal{B}))}}\ge \frac{c_5\sqrt{\log n}}{n}.
\end{align}
Here, $c_{2}, \cdots, c_5>0$ are universal constants.

\section{Gaussian Influence: An Interpretation}\label{appendix_2}

Let \(\gamma = \gamma_n\) denote the standard Gaussian measure on \(\mathbb{R}^n\).  
For a set \(\mathcal{A} \subseteq \mathbb{R}^n\), recall the \(k\)-th Gaussian influence of \(\mathcal{A}\) defined before
\[
I_{k}^{(\gamma)}(\mathcal{A}) = \mathbb{E}_{\gamma}\!\bigl[\mathbf{1}_{\mathcal{A}} \cdot \mathbf{x}_{k}\bigr],
\]
where \(\mathbf{x}_{k}\) is the \(k\)-th coordinate function.  
Although this expression is the direct Gaussian analogue of the discrete influence (see~\eqref{Eq_influence}), it may appear unmotivated at first glance. The purpose of this brief section is to provide an intuitive justification for this definition and to demonstrate that it captures the same notion of coordinate sensitivity as its discrete counterpart.

Let \(\gamma_1\) be the standard Gaussian measure on \(\mathbb{R}\), and let \(\varphi\) denote its density function. For a Borel-measurable set \(A \subset \mathbb{R}\), define
\[
\gamma_1^{+}(A) = \liminf_{r \to 0^+} \frac{\gamma_1(A + [-r, r]) - \gamma_1(A)}{r}.
\]
For any Borel-measurable set \(\mathcal{A} \subset \mathbb{R}^n\) and any \(x = (x_1, \dots, x_n) \in \mathbb{R}^n\), define the section of \(\mathcal{A}\) at \(x\) along the \(k\)-th coordinate as
\[
\mathcal{A}_{k}^x := \bigl\{y \in \mathbb{R} : (x_1, \dots, x_{k-1}, y, x_{k+1}, \dots, x_n) \in \mathcal{A}\bigr\}.
\]
Keller, Mossel, and Sen~\cite{KMS2014} proposed the following alternative definition of the Gaussian influence of the \(k\)-th coordinate on \(\mathcal{A}\):
\[
\hat{I}_{k}^{(\gamma)}(\mathcal{A}) = \mathbb{E}_{\gamma_{n-1}}\!\bigl[\gamma_{1}^{+}(\mathcal{A}_{k}^{x})\bigr],
\]
where \(\gamma_{n-1}\) is the standard Gaussian measure on \(\mathbb{R}^{n-1}\). This definition is more geometrically intuitive, as it quantifies the boundary measure of the sections \(\mathcal{A}_{k}^{x}\) averaged over the remaining coordinates.  
We will later show that these two definitions coincide when \(\mathcal{A}\) is an increasing set, i.e. if $x\le y$, then $\mathbf{1}_{\mathcal{A}}(x)\le \mathbf{1}_{\mathcal{A}}(y)$. 

On the one hand, we have for an increasing set $\mathcal{A}\subset \mathbb{R}^n$
\begin{align}
    \mathcal{A}_{k}^x=\{y\in\mathbb{R}: y\ge t_{k}(x) \},
\end{align}
where $t_k(x)=\inf\{y\in\mathbb{R}: (x_1, \cdots,y,\cdots,x_n)\in \mathcal{A} \}$. Hence, we have
\begin{align}
    \hat{I}_{k}^{(\gamma)}(\mathcal{A}) = \mathbb{E}_{\gamma_{n-1}}\!\bigl[\gamma_{1}^{+}(\mathcal{A}_{k}^{x})\bigr]=\mathbb{E}_{\gamma_{n-1}}\!\bigl[\varphi(t_k(x))\bigr].
\end{align}

On the other hand, we have by Fubini's theorem and partial integration
\begin{align}
    I_{k}^{(\gamma)}(\mathcal{A}) = \mathbb{E}_{\gamma}\!\bigl[\mathbf{1}_{\mathcal{A}} \cdot \mathbf{x}_{k}\bigr]&=\mathbb{E}_{\gamma_{n-1}}\!\bigl[\int_{t_k(x)}^{\infty}x_k\varphi(x_k)\, dx_k\bigr]\nonumber\\
    &=\mathbb{E}_{\gamma_{n-1}}\!\bigl[\varphi(t_k(x))\bigr]= \hat{I}_{k}^{(\gamma)}(\mathcal{A}).
\end{align}

\section{Covariance Formulas for Gaussian Random Variables}\label{appendix_3}
In this section we derive a covariance identity for Gaussian random variables. Although the calculation is elementary, the result is not immediately obvious. In particular, let $\xi$ and $\eta$ be two standard Gaussian random variables with correlation $\rho\in [0, 1]$. Then, we have for any $t, s\in \mathbb{R}$
\begin{align}\label{appendix_3_1}
    \mathrm{Cov}\big(\mathrm{sgn}(\xi-t), \mathrm{sgn}(\eta-s)\big)=4\Big(\mathbb{P}\big( \xi>t, \eta>s   \big)-\mathbb{P}\big( \xi>t \big)\cdot\mathbb{P}\big(  \eta>s\big)\Big).
\end{align}

Below we shall prove \eqref{appendix_3_1}.
Let  
\[
X=\operatorname{sgn}(\xi-t),\qquad Y=\operatorname{sgn}(\eta-s).
\]  
Then  
\[
\mathbb{E} X=1\cdot\mathbb{P}(\xi>t)+(-1)\cdot\mathbb{P}(\xi\le t)
         =\mathbb{P}(\xi>t)-\mathbb{P}(\xi\le t)
         =2\mathbb{P}(\xi>t)-1,
\]  
and similarly  
\[
\mathbb{E} Y=2\mathbb{P}(\eta>s)-1.
\]  
Next observe the point-wise identity  
\[
X Y=\mathbf{1}_{\{\xi>t,\eta>s\}}+\mathbf{1}_{\{\xi\le t,\eta\le s\}}
    -\mathbf{1}_{\{\xi>t,\eta\le s\}}-\mathbf{1}_{\{\xi\le t,\eta>s\}}.
\]  
Taking expectations gives  
\[
\mathbb{E}[X Y]=\mathbb{P}(\xi>t,\eta>s)+\mathbb{P}(\xi\le t,\eta\le s)
               -\mathbb{P}(\xi>t,\eta\le s)-\mathbb{P}(\xi\le t,\eta>s).
\]  

Now compute the covariance:  
\[
\begin{aligned}
\operatorname{Cov}(X,Y)
&=\mathbb{E}[X Y]-\mathbb{E} X\,\mathbb{E} Y\\[4pt]
&=\Bigl[\mathbb{P}(\xi>t,\eta>s)+\mathbb{P}(\xi\le t,\eta\le s)
      -\mathbb{P}(\xi>t,\eta\le s)-\mathbb{P}(\xi\le t,\eta>s)\Bigr]\\[2pt]
&\quad -\bigl[2\mathbb{P}(\xi>t)-1\bigr]\bigl[2\mathbb{P}(\eta>s)-1\bigr].
\end{aligned}
\]  
Expand the product term:  
\[
[2\mathbb{P}(\xi>t)-1][2\mathbb{P}(\eta>s)-1]
=4\mathbb{P}(\xi>t)\mathbb{P}(\eta>s)
-2\mathbb{P}(\xi>t)
-2\mathbb{P}(\eta>s)
+1.
\]  
Use the marginal relations  
\[
\mathbb{P}(\xi\le t)=1-\mathbb{P}(\xi>t),\qquad
\mathbb{P}(\eta\le s)=1-\mathbb{P}(\eta>s)
\]  
to rewrite every “$\le$” probability in terms of “$>$” probabilities.  After substitution the first bracket becomes  
\[
\begin{aligned}
&\ \mathbb{P}(\xi>t,\eta>s)
  +\bigl[1-\mathbb{P}(\xi>t)-\mathbb{P}(\eta>s)+\mathbb{P}(\xi>t,\eta>s)\bigr]\\[2pt]
&-\bigl[\mathbb{P}(\xi>t)-\mathbb{P}(\xi>t,\eta>s)\bigr]
 -\bigl[\mathbb{P}(\eta>s)-\mathbb{P}(\xi>t,\eta>s)\bigr]\\[4pt]
&=\ 4\mathbb{P}(\xi>t,\eta>s)-2\mathbb{P}(\xi>t)-2\mathbb{P}(\eta>s)+1,
\end{aligned}
\]  
which is identical to the expanded product.  Hence their difference collapses to  
\[
\operatorname{Cov}(X,Y)=4\Bigl[\mathbb{P}(\xi>t,\eta>s)
                           -\mathbb{P}(\xi>t)\mathbb{P}(\eta>s)\Bigr].
\]  
This completes the proof. 
\section{Auxiliary results}
Define  
\[
h_r(t,s)=\frac{1}{\sqrt{1-r^{2}}}\exp\!\Bigl(\frac{rts-\tfrac12(t^{2}+s^{2})r^{2}}{1-r^{2}}\Bigr).
\]

\begin{lemma}\label{Lem_appendix_4}
Let $t\ge 1$, $k\ge 1$.  For every $r\in[0,\,1/(2tk)]$ we have
\[
h_{r}(t,-k)\ge e^{-1}.
\]
\end{lemma}

\begin{proof}
Fix $r\le 1/(2tk)$, then $r\le 1/2$ and $1-r^{2}\ge 3/4$, so
$\frac{1}{1-r^{2}}\le\frac{4}{3}$.  The exponent in $h_{r}(t,-k)$ equals
\[
\frac{-rtk-\tfrac12(t^{2}+k^{2})r^{2}}{1-r^{2}}
\ge-\frac{4}{3}\Bigl(rtk+\tfrac12(t^{2}+k^{2})r^{2}\Bigr).
\]
Since $r\le 1/(2tk)$ and $t,k\ge 1$,
\[
rtk\le\frac{1}{2},\qquad
(t^{2}+k^{2})r^{2}\le\frac{t^{2}+k^{2}}{4t^{2}k^{2}}\le\frac{1}{2},
\]
where the last bound uses $(t^{2}+k^{2})/(t^{2}k^{2})\le 2$.  Hence
\[
rtk+\tfrac12(t^{2}+k^{2})r^{2}\le\frac{1}{2}+\frac{1}{4}=\frac{3}{4},
\]
and therefore
\[
\frac{-rtk-\tfrac12(t^{2}+k^{2})r^{2}}{1-r^{2}}\ge-\frac{4}{3}\cdot\frac{3}{4}=-1.
\]
Together with $1/\sqrt{1-r^{2}}\ge 1$ this gives
\[
h_{r}(t,-k)\ge e^{-1}. \qedhere
\]
\end{proof}
\vskip 2mm
\Large\textbf{Acknowledgment:} The authors are grateful to Professor Wang Ke for her fruitful discussions.

\bibliography{conjecture_KKM}
\bibliographystyle{abbrv}

\end{document}